\documentclass[12pt]{article}

\usepackage[dvipsnames]{xcolor}
\usepackage{graphicx}   
\usepackage{amssymb,mathrsfs,dsfont}

\usepackage{amsthm}
\usepackage{units}
\usepackage{amsmath}

\def\eps{\varepsilon}
\newcommand{\R}{\mathbb{R}}
\newcommand{\N}{\mathbb{N}}

\newcommand{\M}{\mathscr{M}}

\newcommand{\Id}{\mathrm{Id}}

\newcommand{\be}{\begin{equation}}
\newcommand{\ee}{\end{equation}}
\newcommand{\beq}{\begin{equation}}
\newcommand{\eeq}{\end{equation}}

\newcommand{\clf}{common Lyapunov function}

\newtheorem{theorem}{Theorem}[section]
\newtheorem{corollary}[theorem]{Corollary}
\newtheorem{proposition}[theorem]{Proposition}
\newtheorem{lemma}[theorem]{Lemma}
\newtheorem{definition}[theorem]{Definition}

\newtheorem{remark}[theorem]{Remark}
\newtheorem{example}[theorem]{Example}

\title{On universal classes of Lyapunov functions for linear switched systems\thanks{This research was partially supported by  the iCODE institute, research project of the Idex Paris-Saclay.}} 

\author{Paolo Mason\thanks{Universit\'e Paris-Saclay, CNRS, CentraleSup\'elec,  Laboratoire des signaux et syst\`emes, 91190, Gif-sur-Yvette, France, {\tt paolo.mason@centralesupelec.fr}}, Yacine Chitour\thanks{Universit\'e Paris-Saclay, CNRS, CentraleSup\'elec,  Laboratoire des signaux et syst\`emes, 91190, Gif-sur-Yvette, France, {\tt yacine.chitour@centralesupelec.fr}}, and Mario Sigalotti\thanks{Laboratoire Jacques-Louis Lions, CNRS, Inria, Sorbonne Universit\'e, Universit\'e de Paris, France, {\tt mario.sigalotti@inria.fr}}}

\begin{document}

\maketitle

\begin{abstract}
In this paper we discuss the notion of universality for classes of candidate common Lyapunov functions for linear switched systems. 
On the one hand, we prove that a family of absolutely homogeneous functions is universal as soon as it approximates arbitrarily well every convex absolutely homogeneous function for the $C^0$ topology of the unit sphere. On the other hand, we prove several obstructions for a class to be universal, showing, in particular,
 that families of piecewise-polynomial continuous functions whose construction involves at most $l$ polynomials of degree at most $m$ (for given positive integers $l,m$) cannot be universal.
\end{abstract}


\section{Introduction}

Common Lyapunov functions constitute the most popular and powerful tool for the stability analysis of switched systems. 
Roughly speaking, the use of common Lyapunov functions for stability analysis gathers the global behavior of the system and allows to bypass the explicit analysis of single trajectories, which may be extremely complex.
Yet, looking for a common Lyapunov function may be a nontrivial task as stability cannot always be checked by means of Lyapunov functions in a simple form, for instance within the class of quadratic forms. 
Given a family of systems (e.g., the family of all linear switched systems), classes of functions large enough to include a Lyapunov function for each globally asymptotically stable system are called \emph{universal}~\cite{blanchini} and a result establishing the existence of such a class is  called a \emph{converse Lyapunov theorem}.
The literature dealing with converse Lyapunov theorems, starting from the works by Massera and Kurzweil in the 1950s (see e.g.~\cite{massera,kurzweil,wilson,sontag-converse,hante})  is quite rich. The results concerning the existence of smooth Lyapunov functions for nonlinear systems with global asymptotic stability properties require the development of rather sophisticated techniques.  Concerning robust asymptotic stability 
with respect to 
a closed invariant set in presence of perturbation terms, converse Lyapunov theorems have been derived in~\cite{sontag-converse}. In the context of switched systems (even in a nonlinear setting) such results establish the equivalence between the global uniform asymptotic stability and the existence of a smooth Lyapunov function. For linear switched systems the construction of a Lyapunov function is much more direct and natural due, essentially, to the homogeneous nature of the system and the equivalence between asymptotic and exponential stability (see e.g.~\cite{DM}). Furthermore, in the linear case and even for the more general class of uncertain systems, it is well-known that the families of piecewise quadratic functions, polyhedral functions, and homogeneous polynomials are universal~\cite{molch1,molch2,molch-scl,blanchini}. On the other hand, for every positive integer $m$, the family of polynomials of degree less or equal than $m$ is not universal even for the simple class of two-dimensional linear switched systems with two modes~\cite{bcm}. Similarly, it is well accepted in the research community (although, to the authors' knowledge, no explicit proof is available)  that families of piecewise quadratic and polyhedral functions whose construction involves a uniformly bounded number of quadratic or linear functions cannot be universal. For this reason, all numerical methods  investigating the existence of Lyapunov functions within these classes are affected by a certain degree of conservativeness. 

The contribution of this paper is twofold. First, we provide a general sufficient condition for a class of functions to be universal (Proposition~\ref{prop-approx}), which is a formalization of fundamental ideas already present in \cite{molch1,molch2,molch-scl,blanchini}.  As a corollary, we recover the universal classes of functions obtained in these references. 
We next derive the main results of this  paper, which provide some necessary conditions for the universality of classes of functions. The first one, Theorem~\ref{no-universal}, is an abstract result which applies to families of real-valued  functions that are  analytic  outside
the origin. The fact that polynomials with a uniform bound $m$ on their degree do not form a universal class~\cite{bcm} follows as a simple consequence of this result. Finally, Theorem~\ref{prop-suivante} states that families of piecewise-polynomial continuous functions whose construction involves at most $l$ polynomials of degree at most $m$ (for given positive integers $l,m$) cannot be universal.

\section{Universal classes of common Lyapunov functions}\label{s:converse-lyapunov}

We consider linear switched systems of the form 
\[
\dot x(t)=A(t)x(t),\qquad t\geq 0,\quad x\in\R^n,\eqno{(\Sigma_{\M})}
\]
where the switching law $A$ is an arbitrary function  belonging  to the space 
$L^\infty(\mathbb{R}_+,\M)$
of measurable functions 
from $\R_+=[0,+\infty)$ to
a bounded subset  $\M$ of the set of $n\times n$ matrices, denoted by $M_n(\R)$.
We use $\Phi_A(t,s)$ to denote the fundamental matrix from $s$ to $t$ for $(\Sigma_\M)$ associated with the switching law $A$ so that every solution of $(\Sigma_\M)$ can be written as $x(t) = \Phi_A(t,0)x(0)$.
Notice that $\Phi_A(t,s)$ exists for every $A\in L^\infty(\mathbb{R}_+,\M)$ and every $s,t\ge 0$ 
(see e.g. \cite[Theorem~3.7]{BressanPiccoli}).
We are interested in the following  uniform stability properties.
\begin{definition}\label{def:stab}
The switched system $(\Sigma_\M)$ is said to be
\begin{itemize}
\item \emph{uniformly stable} 
if there exists $C>0$ such that, for every 
switching law $A$ and $t\geq 0$, $\|\Phi_A(t,0)\|\le C $;
\item \emph{uniformly exponentially stable}
if there exist $C,\gamma>0$ such that, for every switching law 
$A$ and $t\geq 0$, $\|\Phi_A(t,0)\|\le C e^{-\gamma t}$.
\end{itemize}
\end{definition}
Stability in the previous senses may be assessed through common Lyapunov functions, defined below.

\begin{definition} \label{d-CLF} 
We say that a continuous function $V:\R^n\longrightarrow \R_+$ is a \emph{nonstrict
common Lyapunov function} 
for $(\Sigma_{\M})$ if it is \emph{positive definite}, that is, $V(0)=0$ and $V(x)>0$ for every $x\neq 0$, and $V$  is non-increasing along each trajectory of $(\Sigma_{\M})$. If, moreover, $V$ is strictly decreasing along each nonzero trajectory of $(\Sigma_{\M})$, we say that $V$ is a \emph{common Lyapunov function} for $(\Sigma_{\M})$.
\end{definition}

\begin{remark}\label{rem:riscala}
If $V$ is a (possibly nonstrict) common Lyapunov function and $\varphi:\R_+\to\R_+$ is continuous, strictly increasing, and such that $\varphi(0)=0$,
then $\varphi\circ V$ is also a (nonstrict) common Lyapunov function. In particular, 
the positive multiple of a common Lyapunov function
is a common Lyapunov function and if there exists an absolutely homogeneous common Lyapunov function\footnote{Given $\alpha>0$,  a function $V:\R^n\to \R$ is said to be \emph{absolutely homogeneous of degree $\alpha$} if $V(\lambda x)=|\lambda|^\alpha V(x)$ for every $x\in \R^n$ and $\lambda\in \R$. If $\alpha= 1$ we simply say that $V$ is \emph{absolutely homogeneous}.},
then for every $\alpha>0$ there exists an absolutely homogeneous common Lyapunov function of degree $\alpha$.
\end{remark}

We state here 
the classical direct Lyapunov theorem in the  
linear switched case   (see, e.g.,  \cite[Theorem 2.1]{liberzon-book} or \cite[Theorem 4.2]{hafstein} for a formulation involving merely continuous Lyapunov functions).
\begin{theorem}\label{th:Lyap-direct}
A linear switched system $(\Sigma_{\M})$ admitting a nonstrict \clf\ is uniformly stable. If there exists a strict common Lyapunov  function for $(\Sigma_\M)$, then the latter is uniformly exponentially stable.
\end{theorem}
\begin{remark}
In Definition~\ref{d-CLF} we do not require Lyapunov functions to be proper, i.e., such that the corresponding sublevel sets are compact. This is motivated by the fact that the existence a Lyapunov function in the sense of Definition~\ref{d-CLF} implies local stability of $(\Sigma_\M)$ and hence, by linearity of the system, its global stability. 
As a matter of fact, given a possibly non-proper Lyapunov function $V$ for $(\Sigma_\M)$, a proper Lyapunov function may always be constructed by considering the Minkowski functional associated with a sublevel set of $V$, see e.g.~\cite{blanchini1999set}.
\end{remark}
In case the strict common Lyapunov  function $V:\mathbb{R}^n\to\mathbb{R}_+$ in the above theorem is of class $\mathcal{C}^1$ on $\mathbb{R}^n\setminus\{0\}$,  a standard test for checking the strict decrease of $V$ along non-trivial trajectories of $(\Sigma_\M)$ goes as follows:
\begin{equation}\label{eq:strict-decrease}
\nabla V(x)^\top Mx<0, \qquad \forall M\in\M, \quad \forall x\in \mathbb{R}^n\setminus\{0\}.
\end{equation}

\begin{definition}
A set $\mathcal{P}$ of functions from $\R^n$ to $\R$ is a \emph{universal class of Lyapunov functions} if 
for every bounded set $\M\subset M_n(\R)$ such that $(\Sigma_\M)$ is uniformly 
exponentially stable there exists a common Lyapunov function for $(\Sigma_\M)$ 
in $\mathcal{P}$.
\end{definition}
An equivalent formulation of the universality of a class  $\mathcal{P}$ is that \emph{the converse Lyapunov theorem holds true within $\mathcal{P}$}.

As mentioned in the introduction, the theoretical construction of a common Lyapunov function for linear switched systems can be easily obtained. For instance, a locally Lipschitz continuous Lyapunov function may be defined as 
\begin{equation}\label{basic-LF}
V(x) =\sup_{A\in L^\infty(\mathbb{R}_+,\M)}
\int_0^{+\infty} \|\Phi_A(t,0) x\| d t,
\end{equation}
and may be regularized outside the origin by convolution with a smooth function. 
This classical construction leads to the following result (see e.g.~\cite{meilakhs1979design,DM,bcm}).
\begin{proposition}\label{prop:firstuniversal}
Let  $\alpha\ge 1$ and $\mathcal{P}$ be the class of   
convex absolutely homogeneous functions of degree $\alpha$ on $\R^n$ that are positive and smooth on $\R^n\setminus\{0\}$. Then $
\mathcal{P}$ is universal.
Moreover, 
  for every bounded set $\M\subset M_n(\R)$ such that $(\Sigma_\M)$ is uniformly 
exponentially stable, there exists $\varepsilon>0$ and a common Lyapunov function $V\in \mathcal{P}$ for $(\Sigma_\M)$  such that
$\nabla V(x)^{\top} Mx\leq -\|x\|^{\alpha}$ for every $x\in\R^n\setminus \{0\}$ and $M\in\M$.
\label{p-basic-universal}
\end{proposition}
Note that a globally smooth Lyapunov function may be constructed by classical regularization techniques developed in a nonlinear setting (see e.g.~\cite{kurzweil,sontag-converse}), at the price of 
losing homogeneity. 

Similar to Proposition~\ref{p-basic-universal}, the following 
converse Lyapunov result links the uniform stability of $(\Sigma_\M)$ with the existence of a nonstrict common Lyapunov function, see e.g. \cite[Theorem 2.2]{jungers}.
\begin{proposition}
Assume that $\M\subset M_n(\R)$ is bounded and $(\Sigma_\M)$ is uniformly stable. Then the function 
\begin{equation}
\label{eq:basicLyap}
V(x)=\sup_{t\geq 0,A\in L^\infty(\mathbb{R}_+,\M)} \|\Phi_A(t,0) x\|
\end{equation}
is absolutely homogeneous 
and a nonstrict common Lyapunov function for $(\Sigma_\M)$. 
\label{p-hyperbasic-universal}
\end{proposition}
Due to Proposition~\ref{p-basic-universal}, the continuous differentiability outside the origin of the common Lyapunov function 
is not a restrictive assumption when checking the uniform exponential stability of a linear switched system. On the other hand, the uniform stability of $(\Sigma_{\M})$ does not always imply the existence of a $\mathcal{C}^1$ nonstrict common Lyapunov function (see e.g. \cite[Example 3]{chitour-gaye-mason}).
Furthermore, even in case of uniform exponential stability, it may be useful to provide a criterion to ensure
 the existence of a common Lyapunov function in a class of non-differentiable functions, such as piecewise linear or piecewise quadratic ones.
For these reasons we introduce below a criterion which generalizes Equation~\eqref{eq:strict-decrease} and characterizes the family of (possibly nonstrict) common Lyapunov functions in a nonsmooth setting. We refer to \cite[Proposition 1]{bacciotti-ceragioli} for a similar result in the context of differential inclusions. 

We need the following preliminary result, which expresses the variation of a convex function $V$ along a trajectory in terms of the subdifferential of $V$. Recall that the subdifferential $\partial V(x)$ at a point $x\in \R^n$ is defined as
\[\partial V(x) = \{l\in \R^n\mid l^{\top}(y-x)\leq V(y)-V(x),\quad\forall y\in \R^n\}.\]
The proof of the lemma is similar to that of \cite[Lemma 1]{bacciotti-ceragioli} and is provided here for completeness.
\begin{lemma}\label{derivative}
Let $V:\mathbb{R}^n\to\mathbb{R}$ be a convex function and $\varphi:I \to \mathbb{R}^n$ be an absolutely continuous function, with $I\subseteq \mathbb{R}$ an open interval.
Then $V\circ \varphi$ is absolutely continuous and it holds 
\[\frac{d}{dt}V(\varphi(t)) = l^{\top}\dot\varphi(t),\quad \forall l\in \partial V(\varphi(t)),\quad \hbox{for a.e. }t\in I.\]
\end{lemma}
\begin{proof}
As $V$ is convex, $V$ is Lipschitz and the composition $V\circ \varphi$ is absolutely continuous. Hence for almost every $t\in I$ the derivatives of both $\varphi$ and $V\circ \varphi$ are well-defined. By definition of subdifferential, for every $t,s\in I$ and $l\in \partial V(\varphi(t))$ we have 
\[l^{\top} (\varphi(s)-\varphi(t)) \leq V(\varphi(s))-V(\varphi(t)).\]
We deduce that
\begin{align*}
\frac{d}{dt}V(\varphi(t)) & = \lim_{s\to t^+} \frac{V(\varphi(s))-V(\varphi(t))}{s-t}\\
& \geq l^{\top} \lim_{s\to t^+} \frac{\varphi(s)-\varphi(t)}{s-t}\\
& = l^{\top} \dot\varphi(t)
\end{align*}
holds true  for almost every $t\in I$ and for every $l\in \partial V(\varphi(t))$. Similarly, taking the limit as $s\to t^-$, we obtain that $\frac{d}{dt}V(\varphi(t))\leq  l^{\top} \dot\varphi(t)$ for almost every $t\in I$ and for every $l\in \partial V(\varphi(t))$. This concludes the proof of the lemma.
\end{proof}
Here follows an adaptation to the nonsmooth setting of the characterization of common Lyapunov function. 
\begin{proposition} 
\label{Lyap-nonsmooth}
Let $\M$ be a bounded subset of $M_n(\R)$ and $V:\R^n\to\R_+$ be a convex positive definite 
function. 
Then $V$ is a nonstrict common Lyapunov function  for $(\Sigma_{\M})$ if and only if 
\begin{equation}
\label{Lyap-ineq}
l^{\top} Mx\leq 0,\qquad \forall x\in\R^n\setminus \{0\},\forall l\in \partial V(x),\forall M\in\M.
\end{equation}
 Moreover, if the inequality in \eqref{Lyap-ineq} is strict then $V$ is a common Lyapunov function  for $(\Sigma_{\M})$.
 \end{proposition}
 \begin{proof} 
The second part of the proposition and the \emph{if} implication in the first part directly follow from Lemma~\ref{derivative}. We are left to show that if $V$ is a nonstrict common Lyapunov function  for $(\Sigma_{\M})$, then the inequality~\eqref{Lyap-ineq} holds true. By contradiction, suppose that there exist $x\in\R^n, l\in \partial V(x)$, and $M\in\M$ such that $l^{\top} Mx > 0$. By~\cite[Theorem~25.6]{rockafellar} one may find a differentiability point $y$ of $V$ such that the pair $(y,\nabla V(y))$ is arbitrarily close to $(x,l)$. In particular we may assume $\nabla V(y)^\top My>0$, that is, $V$ is increasing at $t=0$ along the trajectory $t\mapsto e^{tM}y$, leading to a contradiction.
\end{proof}

\section{
Sufficient conditions for universality
}\label{s:universal-lyapunov}

Given a linear switched system $(\Sigma_\M)$, the family $\mathcal{P}$ identified by Proposition~\ref{p-basic-universal} is 
too broad to admit a tractable parameterization, suitable 
for investigating numerically 
the existence of a Lyapunov function. With this goal in mind, interesting candidate classes $\mathcal{P}$ are those parametric families of functions for which the property of being positive definite and strictly decreasing along all admissible dynamics can be translated into numerically verifiable algebraic relations or inequalities (e.g., linear matrix inequalities). It is well-known that piecewise-quadratic, polynomial, and polyhedral functions represent examples of such families~\cite{molch1,molch2,molch-scl,blanchini}.

We next provide a general sufficient condition for a class $\mathcal{P}$ to be universal, namely its density in the class of convex absolutely homogeneous functions for the topology of uniform convergence on compact sets. 
The proof of the sufficient condition exploits
the fact, specific to convex functions defined on compact sets, that being close in the uniform norm is equivalent to possessing ``close'' subdifferentials.

\begin{proposition}\label{prop-approx}
Let $\mathcal{P}$ be
 a subset of  the family of convex absolutely homogeneous functions 
 from $\R^n$ to $\R_+$. Assume that for every convex absolutely homogeneous function $V:\R^n\to\R_+$ 
 and every $\delta>0$ there exists a function $W$ in $\mathcal{P}$ such that $\|W(x)-V(x)\|\leq \delta $ for every $x$ in the unit sphere $S^{n-1}$ of $\R^n$. 
Then $\mathcal{P}$ is a universal class of Lyapunov functions.
\end{proposition}
\begin{proof}
Let $(\Sigma_{\M})$ be uniformly exponentially stable. 
Let $V$ be the absolutely homogeneous 
common Lyapunov function provided by Proposition~\ref{p-basic-universal}.
In particular, $V$ is smooth on $\R^n\setminus\{0\}$.
In order to prove the proposition, it is enough to show that any convex  absolutely homogeneous function close enough to $V$ on $S^{n-1}$ in uniform norm is 
itself 
a Lyapunov function for $(\Sigma_{\M})$. 

We proceed by contradiction: we assume that there exists a sequence of convex absolutely homogeneous functions $(W_k)_{k\in\mathbb{N}}$ converging uniformly to $V$ on $S^{n-1}$ as $k$ goes to infinity and such that each $W_k$ is not strictly decreasing along at least one trajectory of the system. In particular the derivative of $W_k$ along such a trajectory is nonnegative on a set of times of positive measure. By Lemma~\ref{derivative} and absolute homogeneity of $W_k$, we  deduce that there exist $x_k\in S^{n-1}$ and $M_k\in\M$ such that, for every fixed $l_k\in \partial W_k(x_k)$, one has $l_k^{\top} M_k x_k\geq 0$.
By compactness, we may assume that  $x_k$ tends to $\bar x\in S^{n-1}$ as $k$ goes to infinity. Then, by~\cite[Theorem~24.5]{rockafellar},
$l_k$ converges to $\nabla V(\bar x)$, so that, by boundedness of $\M$, $\limsup_{k\to\infty} \nabla V(\bar x)^{\top} M_k\bar x = \limsup_{k\to\infty} l_k^{\top} M_k x_k \geq 0$.
However, it follows by the choice of $V$ and Proposition~\ref{p-basic-universal} that 
 $\nabla V(\bar x)^{\top} M_k\bar x \leq -1$, yielding a contradiction. 
\end{proof}

\begin{remark}
By the absolute homogeneity property, the statement of Proposition~\ref{prop-approx} could be equivalently reformulated by fixing $\delta=1$. 
\end{remark}

As an application of the previous result, 
two classical examples of universal classes of Lyapunov functions (cf.~\cite{blanchini,molch1,molch2,molch-scl}) are recalled in the following corollary.

\begin{corollary}
\label{c-universal}
The family of polyhedral functions 
$\{\max_{k=1,\dots,N} |l_k^{\top} x|\mid l_k\in\mathbb{R}^n,\,N\in\mathbb{N}\}$
and that of homogeneous sums of squares 
$\{\sum_{k=1}^N (l_k^{\top} x)^{2d}\mid l_k\in\mathbb{R}^n,\,d,N\in\mathbb{N}\}$ 
are universal classes of Lyapunov functions. 
\end{corollary}
\begin{proof}
Let $V$ be a convex absolutely homogeneous function. 
Let $(x_i)_{i\in\N}$ be a dense sequence in $S^{n-1}$, and $l_i\in \partial V(x_i)$. We consider the increasing sequence of absolutely homogeneous functions  
defined by
\[W_i(x) =\max_{j=1,\dots, i} |l_j^{\top} x|,\quad \forall x\in\mathbb{R}^n.\]
Observe that each $W_i$ is convex and $W_i(x)\leq V(x)$ for every $x\in \R^{n}$. Indeed 
\begin{align*}
|l_j^{\top} x| & =  \max\{l_j^{\top} x ,l_j^{\top} (-x)\}\\
& = \max\{l_j^{\top} (x-x_j) ,l_j^{\top} (-x-x_j)\} + l_j^{\top} x_j\\
&\leq V(x)-V(x_j)+ l_j^{\top} x_j \\
&= V(x),
\end{align*}
for every positive integer $j$, by definition of subgradient and since $V(x) = V(-x)$ and $V(x_j) = l_j^{\top} x_j$. 
We deduce that $W_i(x_k) = V(x_k)$ for every $k\in \N$ and $i\geq k$, hence $\lim_{i\to \infty} W_i(x_k) = V(x_k)$ and we can apply~\cite[Theorem~10.8]{rockafellar} to conclude that the sequence of functions $W_i$ converges to $V$ uniformly on 
$S^{n-1}$. 
By applying Proposition~\ref{prop-approx} we get that the family of polyhedral functions is   a universal class of Lyapunov functions.

Let us now consider the absolutely homogeneous functions  
\[Z_i(x) =\left(\sum_{j=1}^i|l_j^{\top} x|^{2i}\right)^{\frac{1}{2i}}.\]
The function $Z_i$ is convex since it is the composition of the $2i$-norm on $\R^i$, i.e., $\|y\|_{2i}=\left(\sum_{j=1}^i y_j^{2i}\right)^{\frac{1}{2i}}$ for $y\in \R^i$, with the linear function from $\R^n$ to $\R^i$ mapping $x$ to $(l_1^{\top}x,\dots,l_i^{\top}x)^{\top}$.
Moreover it is immediate to see that $W_i(x)\leq Z_i(x) \leq i^{\frac{1}{2i}}W_i(x)$, and in particular \[\|Z_i(x)-V(x)\|\leq \|Z_i(x)-W_i(x)\|+\|W_i(x)-V(x)\|\]
tends to zero
uniformly on $S^{n-1}$ as $i$ goes to infinity.
By applying again Proposition~\ref{prop-approx}, it follows that the family of homogeneous sums of squares   is   a universal class of Lyapunov functions.
\end{proof}

\begin{remark}
According to Remark~\ref{rem:riscala},
the first part of Corollary~\ref{c-universal} remains valid if one replaces the piecewise linear functions $\max_{k=1,\dots,N} |l_k^{\top} x|$ with the functions $\big(\max_{k=1,\dots,N} |l_k^{\top} x|\big)^q=\max_{k=1,\dots,N} |l_k^{\top} x|^q$, for any given $q>1$. In particular, for $q=2$, we have that the family of piecewise quadratic functions is a universal class of Lyapunov functions.
\end{remark}

\begin{remark}
The proof of Proposition~\ref{prop-approx} relies on the fact that, whenever a linear switched system is uniformly 
exponentially stable, there exists a common  Lyapunov function which is convex.
In  the classical construction \eqref{basic-LF}, convexity and homogeneity of the Lyapunov function are direct consequences of the convexity and homogeneity of the map $x_0\mapsto \|\Phi_A(t,0)x_0\|$ for given $t\geq 0$ and $A\in L^\infty(\mathbb{R}_+,\M)$. 
In the nonlinear case, the
homogeneity property can be recovered from a homogeneity assumption on the vector fields. Similarly, the convexity property can be imposed as an additional requirement. 
Proposition~\ref{prop-approx}, and hence Corollary~\ref{c-universal},
can then be extended to any class of nonlinear switched systems 
\[\dot{x}(t)=f_{\sigma(t)}(x(t)),\qquad  \sigma\in L^\infty(\R_+,\Sigma),\] 
with $\Sigma\subset \R^m$ a bounded set of parameters, satisfying the following conditions:
\begin{itemize}
\item $f_\sigma$ is  
homogeneous of degree one for every $\sigma\in \Sigma$, that is $f_{\sigma}(\lambda x) = \lambda f_{\sigma}(x)$ for every $x\in \R^n$ and $\lambda\in \R$,
\item denoting by $x(t,x_0,\sigma(\cdot))$ the solution at time $t$ of the system starting at $x_0$ and corresponding to the switching law $\sigma\in L^\infty(\R_+,\Sigma)$, the function $x_0\mapsto \|x(t,x_0,\sigma(\cdot))\|$ is convex,
\item for every $R>0$, $\{f_\sigma|_{B(0,R)}\mid  \sigma\in \Sigma\}$ is a  bounded  
subset of $\mathcal{C}(B(0,R),\R^n)$.
\end{itemize} 
The last condition replaces the boundedness of $\M$
which is required in the proofs of both Propositions~\ref{prop:firstuniversal} and \ref{prop-approx}.
\end{remark}

\section{
Necessary conditions for universality
}

Next, we provide restrictions on the classes of functions which may be candidate to be universal. For this purpose, we introduce the following technical result.
\begin{lemma}
\label{l-stability-extension}
Let $\M_1,\M_2$ be bounded 
subsets of $M_n(\R)$ and assume that $(\Sigma_{\M_1})$ is uniformly stable. For $\nu>0$, denote by  $\M_2^{\nu}$ the set  of matrices of the form $M-\nu\Id_n$
for $M\in \M_2$, where $\Id_n$ is the $n\times n$ identity matrix. Set $\M=\M_1\cup\M_2^{\nu}$. Then, the switched system $(\Sigma_{\M})$ is uniformly stable for $\nu>0$ large enough.
\end{lemma}
\begin{proof}
Proposition~\ref{p-hyperbasic-universal} guarantees the existence of a convex 
absolutely homogeneous
nonstrict Lyapunov function $V$ for $(\Sigma_{\M_1})$. 
Intuitively speaking, the lemma follows from the fact that, for $\lambda$ large enough, the vectors $Mx$, with $M\in \M_2^{\lambda}$ and $x\in \R^n$, point towards the interior of the sublevel set $V^{-1}([0,V(x)])$. Let us formalize this idea.
Since $l^{\top}x=V(x)$ whenever $l\in \partial V(x)$, by boundedness of $\M_2$ and of $\cup_{x\in V^{-1}(1)}\partial V(x)$~\cite[Theorem~24.7]{rockafellar}, for $\lambda>0$ large enough one has $l^{\top} (M-\lambda \Id_n)x=l^{\top} Mx-\lambda l^{\top}x =l^{\top} Mx-\lambda  < 0$ for every $M\in \M_2$, $x\in V^{-1}(1)$, and $l\in \partial V(x)$.
The result is then an immediate consequence of Proposition~\ref{Lyap-nonsmooth}. 
\end{proof}
\begin{theorem}
\label{no-universal}
Let $n\geq 2$ and $\mathcal{P}$ be a compact 
subset of the space of 
continuous
functions from $\mathbb{R}^n$ to $\mathbb{R}$
that are analytic on $\mathbb{R}^n\setminus\{0\}$,
endowed with the topology of uniform convergence on bounded sets. 
Assume that $\mathcal{P}$ does not  
contain the zero function. 
Then $\mathcal{P}$  cannot be  a universal class of Lyapunov functions.
\end{theorem}
\begin{proof}
We start by showing the 
theorem
in the case $n=2$. 
We proceed by contradiction, assuming that every uniformly exponentially stable switched system in the case where $\M$ consists of two matrices in $M_2(\R)$ admits a Lyapunov function in $\mathcal{P}$.
We consider a switched system corresponding to $\M^0= \{M^0_1,M^0_2\}\subset M_2(\R)$, where $M^0_1,M^0_2$ are Hurwitz, the corresponding trajectories rotate clockwise around the origin, the system is uniformly stable, but not attractive, and 
starting from every initial nonzero condition there exists a unique periodic trajectory, with four switches per period. The existence of such a system is obtained in~\cite[Theorem~1]{sw-balde2}, where it corresponds to the case $\mathbf{S4}$,  $\mathcal{R}=1$.
In particular, there exist 
$t_1,t_2>0$ such that $e^{t_1 M^0_1}e^{t_2 M^0_2}$ has an eigenvalue equal to $-1$, corresponding to an eigenvector $x_0$. Set $T=t_1+t_2$  and consider the switched systems associated with $\M^{\varepsilon}=\{M^{\varepsilon}_1,M^{\varepsilon}_2\}\subset M_2(\R)$, where $M^{\varepsilon}_i=M^0_i - \varepsilon\Id_2$ for $i=1,2$. For $\varepsilon \geq 0$ we consider the $T$-periodic switching sequence $A^{\varepsilon}(\cdot)$ which takes values $M^{\varepsilon}_1$ for $t\in [0,t_1)$ and $M^{\varepsilon}_2$ for $t\in [t_1,T)$.

Since every trajectory $x(\cdot)$ of $(\Sigma_{\M^{\varepsilon}})$ can be written as $t\mapsto e^{-\varepsilon t}y(t)$ where $y(\cdot)$ is a trajectory of $(\Sigma_{\M})$, then $(\Sigma_{\M^{\varepsilon}})$ is uniformly exponentially stable for $\eps>0$.
Hence, by assumption,  it admits a Lyapunov function $V_{\varepsilon}(\cdot)$ in $\mathcal{P}$. Since the latter is compact, there exists a sequence $(\varepsilon_{k})_{k\in\N}$ converging to zero such that $(V_{\varepsilon_k})_{k\in\N}$ converges to some  $\bar V\in\mathcal{P}$. Moreover, for every $t\geq 0$, 
\begin{align*}
\bar V(\Phi_{A^0}(t,0)x_0)&=\lim_{k\to\infty}V_{\varepsilon_k}(\Phi_{A^{\varepsilon_k}}(t,0)x_0)\\
&\leq \lim_{k\to\infty}V_{\varepsilon_k}(x_0) = \bar V(x_0).
\end{align*}
Since $\bar V(\Phi_{A^0}(2T,0)x_0) = \bar V(x_0)$ we deduce that $\bar V$ is constant along the trajectory $\Phi_{A^0}(\cdot,0)x_0$. The function $t\mapsto\bar V(e^{t M^0_1}x_0)$ is analytic for $x_0\ne0$, being the composition of analytic functions, and it is constantly equal to $\bar V(x_0)$ for $t\in [0,t_1]$. By analyticity, 
$t\mapsto \bar V(e^{t M^0_1}x_0)$ is constant for all $t>0$ and therefore it must be identically equal to $0$ since $\bar V(\lim_{t\to\infty}e^{t M^0_1}x_0) = \bar V (0) = 0$. Since every nonzero point of $\R^2$ may be written as $\mu e^{t M^0_1}x_0$ for some positive $\mu$ and $t$, we deduce that $\bar V$ must be identically zero, contradicting the assumptions on $\mathcal{P}$.

We are left to prove the result for $n>2$. For this purpose we consider 
$\M_1=\{\bar M^0_1,\bar M^0_2\}$ with
\[\bar M^0_i  = \left(\begin{array}{cc}M^0_i & 0\\ 0 & -\Id_{n-2}\end{array}\right),\]
where the matrices $M^0_i $ are defined as above. Let  $\lambda>0$ and $\M_2^{\lambda}$ be given by Lemma~\ref{l-stability-extension} with $\M_2=\{M\in {\rm so}(n)\mid \Vert M\Vert\leq 1\}$, where ${\rm so}(n)$ denotes the space of skew-symmetric $n\times n$ matrices. Define $\bar \M^0=\M_1\cup \M_2^{\lambda}$ and, 
for $\varepsilon> 0$, consider 
the switched system corresponding to $\bar\M^\varepsilon=\bar \M^0-\varepsilon\Id_{n}$.  
 It is clear that $(\Sigma_{\bar \M^\eps})$ is uniformly exponentially stable for every $\varepsilon>0$.
 
Letting $\Pi_{1,2}$  be the $(x_1,x_2)$ plane, i.e., $\Pi_{1,2} = \{x\in \R^n\mid x_3=\dots=x_n=0\}$, we notice that, starting from every $\bar x \in \Pi_{1,2}$, there exists a periodic trajectory of $(\Sigma_{\bar\M^0})$ lying  on $\Pi_{1,2}$.
The restrictions of functions in $\mathcal{P}$  to $\Pi_{1,2}$  form 
a compact set of functions on $\Pi_{1,2}$ that are analytic outside the origin. As  in the case $n=2$, we prove by contradiction that $\mathcal{P}$ is not universal. Assume that there exists a sequence $(V_{\varepsilon_k}(\cdot))_{k\in\N}$ of Lyapunov functions in  $\mathcal{P}$ for  $(\Sigma_{\bar\M^{\eps_k}})$ converging to $\bar V\in\mathcal{P}$. We can show as before that $\bar V$ is equal to $0$ on $\Pi_{1,2}$. 
Because of the choice of $\M_2$ and by construction of $\bar\M^0$,  every $1$-dimensional linear subspace of $\R^n$ may be reached in finite time from $\Pi_{1,2}$ via a trajectory of $(\Sigma_{\bar\M^0})$. Since $\bar V$ is non-increasing along such a trajectory,  we deduce  that $\bar V\equiv 0$ on $\R^n$, obtaining a contradiction.
\end{proof}
\begin{remark}\label{rem:Zero}
The assumption that the zero function is not in $\mathcal{P}$ cannot be removed from the hypotheses of Theorem~\ref{no-universal}. Indeed, consider 
the subset of polynomial functions 
made of the zero polynomial and, for every $N\ge 0$,  the polynomials of degree $N$ with 
absolute value of the coefficients upper bounded by a positive constant $c_N$, chosen in such a way that the supremum 
on the ball of radius $N$ of
the polynomial 
is 
 less than or equal to $1/(N+1)$. 
Since the class $\mathcal{P}$ contains a multiple of any polynomial, it is universal by Corollary~\ref{c-universal}.  It is also compact since
every  sequence in $\mathcal{P}$ admits  a subsequence with either degree going to infinity or constant degree. In the former case, the subsequence converges to zero for the topology of uniform convergence on bounded sets, while in the latter one the coefficients are uniformly bounded and hence the sequence admits a further converging subsequence. 
\end{remark}

\begin{example}
Consider the class of absolutely homogeneous functions of degree two 
\[\mathcal{P}=\left\{x\mapsto e^{\frac{x^\top Q_1 x}{\|x\|^2}}x^\top Q_2 x\mid \|Q_1\|\le 1,\,\|Q_2\|=1\right\}.\] 
The level sets of each element of $\mathcal{P}$ 
are obtained by deforming those of the  quadratic functions $x\mapsto x^\top Q_2 x$ by a positive $0$-homogeneous term. 

As each function in $\mathcal{P}$ is analytic outside the origin, by Theorem~\ref{no-universal} the class $\mathcal{P}$ is not universal, despite being richer than that of quadratic functions.
\end{example}

As a consequence of Theorem~\ref{no-universal} we obtain the following corollary which provides, in particular, a partial counterpart to Corollary~\ref{c-universal} for homogeneous polynomial functions. Namely, we recover that, if we impose a uniform bound on the degree, such functions do not form a universal class of Lyapunov functions, as already established in \cite{bcm}.
\begin{corollary}\label{cor-polynomials}
Let $n\geq 2$ and $\mathcal P$ be a finite-dimensional vector subspace of the space of continuous functions from $\mathbb{R}^n$ to $\mathbb{R}$ that are analytic on $\mathbb{R}^n\setminus\{0\}$. Then $\mathcal{P}$ is not universal.
In particular, for every positive integer $m$, the set of polynomial functions of degree at most $m$ from $\mathbb{R}^n$ to $\mathbb{R}$ is not  a universal class of Lyapunov functions.
\end{corollary}
\begin{proof}
Let $\{f_1,\dots,f_N\}$ be a basis of $\mathcal P$. The linear map $\varphi:\R^N\to \mathcal{P}$ defined as $\varphi(\alpha) = \sum_{i=1}^N\alpha_i f_i$  is continuous 
when $\mathcal{P}$ is endowed with the topology of uniform convergence on compact sets.
Indeed, on every compact set $K\subset\R^n$ and for every $\alpha^1,\alpha^2\in\R^N$, one has 
\[\|\varphi(\alpha^1)|_K-\varphi(\alpha^2)|_K\|_\infty
\leq \|\alpha^1-\alpha^2\|\max_{i=1,\dots,N}\|f_i|_K\|_\infty.\] 
Hence the image of the unit sphere via the map $\varphi$ is a compact set $\bar{\mathcal{P}}\subset \mathcal{P}$. Furthermore,  $\bar{\mathcal{P}}$ does not contain the zero function because $\{f_1,\dots,f_N\}$ is a linearly independent subset of $\mathcal P$. Applying Theorem~\ref{no-universal} we obtain that $\bar{\mathcal{P}}$ is not universal. As each element of $\mathcal{P}$ is a scalar multiple of an element of $\bar{\mathcal{P}}$ we deduce that  $\mathcal{P}$ is not universal either, concluding the proof of the first part of the corollary. 
Concerning the second part, it is enough to observe that the set of polynomial functions of degree at most $m$ from $\mathbb{R}^n$ to $\mathbb{R}$ is a finite-dimensional vector space of analytic functions.
\end{proof}

\begin{remark}
In order to avoid confusion, let us stress that 
a finite-dimensional space as the set  $\mathcal P$ appearing in the statement of Corollary~\ref{cor-polynomials} obviously contains $\{0\}$. This is not in contradiction with Theorem~\ref{no-universal}, since the latter is applied to $\bar{\mathcal{P}}$ and not to $ \mathcal{P}$ in our proof of the corollary. 
\end{remark}

\begin{example}
Let $n=2$, fix a positive integer $m$, and consider the vector space $\mathcal{P}$ of functions from $\mathbb{R}^2$ to $\mathbb{R}$ defined, in polar coordinates, by
\[(r,\theta)\mapsto r a_0 + r\sum_{i=1}^M (a_i \cos(2i \theta)+b_i \sin(2i \theta)),\]
where $a_0,\dots,a_m,b_1,\dots,b_m$ are in $\R$. 
Each function in $\mathcal{P}$ is analytic outside the origin of $\mathbb{R}^2$ and absolutely homogeneous. 

By Corollary~\ref{cor-polynomials} the class $\mathcal{P}$ is not universal.
\end{example}

\begin{example}
\label{numerical-example}
Let $n=2$ and consider the matrices
\[M_1=\begin{pmatrix}-0.1 & -1\\ 1& -0.1 \end{pmatrix},\qquad M_2=\begin{pmatrix}-0.1 & -\alpha\\ \frac 1\alpha& -0.1 \end{pmatrix},\]
where $\alpha>0$. Set $\M=\{M_1,M_2\}$. 
By using the results in \cite{sw-balde2}, we can deduce that $(\Sigma_\M)$ is uniformly exponentially stable if and only if $\alpha_1<\alpha<\alpha_2$, where $\alpha_1\approx 0.819$ and $\alpha_2\approx 1.367$. 

Taken $\alpha\in (\alpha_1,\alpha_2)$, let us focus on the minimal degree of a polynomial homogeneous common Lyapunov function for $(\Sigma_\M)$. 
The existence of a polynomial Lyapunov function homogeneous of a given degree $n$ can be tested using LMIs, as detailed in \cite[Theorems 3.4 and 3.6]{ChesiGarulli}. For $n\leq 32$ we established numerically the existence of homogeneous polynomial Lyapunov functions  for $\alpha\in (\alpha_1,\alpha_2^{(n)})$, where $\alpha_2^{(n)}$ is given in Table~\ref{table-SoS}.

 \begin{table}[h]
\centering
{
\caption{\label{table-SoS}}
}
\begin{tabular}{|c|c|c|c|c|c|}
\hline
 $n$ & 2 & 4 & 6 & 8 & 12\\
\hline
$\alpha_2^{(n)}$ & 1.22 & 1.325 & 1.325 & 1.348 & 1.356\\
\hline\hline
  $n$ & 16 & 20 & 24 & 28 & 32\\
\hline
 $\alpha_2^{(n)}$ & 1.36 & 1.3621 & 1.3634 & 1.3642 & 1.3647\\
\hline
\end{tabular}
\end{table}

In accordance with Corollary~\ref{cor-polynomials}, the minimal degree appears to diverge as $\alpha$ tends to $\alpha_2$. 

\end{example}

The conclusion of Corollary~\ref{cor-polynomials} can be
proved to hold true for functions involving maxima and minima within a finite family of polynomials
such as the class of polyhedral functions $V$ of the form
\begin{align*}
V(x)&=\max\{|l_1^{\top} x|,\dots,|l_N^{\top} x|\}\\
&=\max\{l_1^{\top} x,\dots,l_N^{\top} x,-l_1^{\top} x,\dots,-l_N^{\top} x\},
\end{align*}
with $l_1,\dots,l_N\in\mathbb{R}^n$, where $N$ is fixed. This partial counterpart to Corollary~\ref{c-universal} is a consequence of the following more general result.
\begin{theorem}\label{prop-suivante}
Let $n\ge 2$ and $\mathcal{P}^n_d$ be the family of  polynomial functions in $\mathbb{R}^n$ of degree at most $d$ and $l$ be a positive integer. Consider the family 
\begin{align*}
\mathcal{P}^n_{d,l}&=\{V\in\mathcal{C}(\R^n,\R)\mid \exists V_1,\dots,V_l\in\mathcal{P}^n_d\\
&\mbox{ s.t. } V(x)\in\{V_1(x),\dots,V_l(x)\},\ \forall x\in\mathbb{R}^n\}.
\end{align*}
Then, $\mathcal{P}^{n}_{d,l}$ is not universal.
\end{theorem}
\begin{proof}
We first claim that if 
$\mathcal{P}^{n}_{d,l}$ is  universal, the same is true 
for $\mathcal{P}^{2}_{d,l}$.
Indeed, for every $\M\subset M_2(\R)$ such that $(\Sigma_\M)$ is uniformly exponentially stable, consider $\hat \M\subset M_n(\R)$ given by
\[\hat \M=\left\{\begin{pmatrix}M&0\\ 0& -\Id_{n-2}\end{pmatrix}\mid M\in \M\right\}.\]

If 
$\hat V\in \mathcal{P}^{n}_{d,l}$ is a
common Lyapunov function for $(\Sigma_{\hat \M})$, then $V:\R^2\ni x\mapsto \hat V(x,0)$
is a common Lyapunov function for $(\Sigma_{ \M})$ and $V\in \mathcal{P}^{2}_{d,l}$.

We are left to prove that $\mathcal{P}^{2}_{d,l}$ is not universal. 
Consider the switched systems $(\Sigma_{\M^\eps})$ introduced in the proof of Theorem~\ref{no-universal}, which are uniformly exponentially  stable for $\eps>0$, and only uniformly stable for $\eps=0$.
Assume by contradiction that $\mathcal{P}^2_{d,l}$ is universal and, in particular, that for every $\eps>0$ there exists a Lyapunov function $V^\eps\in\mathcal{P}^2_{d,l}$ for $(\Sigma_{\M^\eps})$. By definition of $\mathcal{P}^2_{d,l}$, for every $\eps>0$ there exist $l$ polynomials $P_1^\eps,\dots,P_l^\eps$ 
of degree at most $d$ such that
$V^\eps(x) \in \{P_1^\eps(x),\dots,P_l^\eps(x)\}$.
Given $j,k\in  \{1,\dots,l\}$, 
we investigate the set of zeroes of the polynomial  $Q_{jk}^\eps$ defined as the homogeneous polynomial corresponding to the terms of maximal degree of  $P_j^\eps-P_k^\eps$. 
 For this purpose, recall that, by the fundamental theorem of algebra, every  homogeneous polynomial $Q$ of positive degree $m$ may be factorized as $Q =  \prod_{k=1}^{m} (\alpha_k x_1 + \beta_k x_2)$, where $\alpha_k,\beta_k\in \mathbb{C}$ for $k=1,\dots,m$, so that its zeroes correspond to the intersection of the unit circle $S^{1}$ with at most $m$ lines through the origin.
Hence, it follows that either $Q_{jk}^\eps\equiv 0$ (i.e., $P_j^\eps\equiv P_k^\eps$) or $Q_{jk}^\eps$ vanishes at most $2d$ times on the unit circle. 
Moreover, for every $\eps>0$, the integer $N=2d\binom{l-1}{2}+1$ is a strict upper bound for the total number of zeroes of $Q_{jk}^\eps$ for $j,k\in  \{1,\dots,l\}$.
Partitioning the circle into $N$ arcs $\mathcal{C}_1,\dots,\mathcal{C}_{N}$ of equal length, for every $\eps>0$ there exists an arc $\mathcal{C}_{n_{\eps}}$ which contains no zero of the nontrivial polynomials $Q_{jk}^\eps$ in its interior. 
Denote by $\mathcal{A}_{n_{\eps}}$ the closed middle third of $\mathcal{C}_{n_{\eps}}$ (see Figure~\ref{fig:middlethird}).
\begin{figure}
    \centering
\includegraphics[width=7cm]{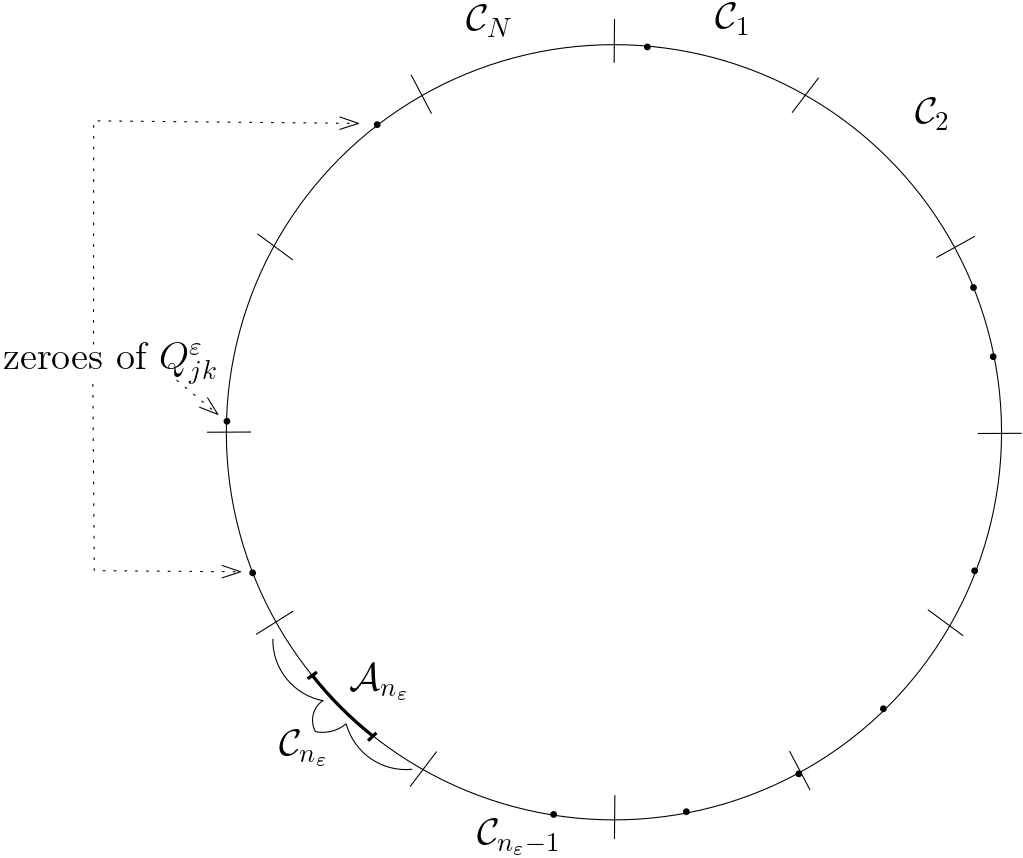}
    \caption{Construction of the arc $\mathcal{A}_{n_{\eps}}$}
    \label{fig:middlethird}
\end{figure}

We next claim that for every $\eps>0$ there exists $\nu_\eps>0$ large enough such that the restriction of $V^\eps$ to the dilated arc $\nu_\eps \mathcal{A}_{n_{\eps}}$ coincides with the restriction to the same arc of one of the polynomials $P_1^\eps,\dots,P_l^\eps$. 
By definition of the function $V^\eps$ and taking into account its  continuity, it is enough to prove that, for every $\eps>0$ there exists $\nu_\eps>0$ large enough such that,  in the interior of the arc $\nu_\eps \mathcal{A}_{n_{\eps}}$, one has for every $j,k\in  \{1,\dots,l\}$ that either $P_j^\eps\equiv P_k^\eps$ or $P_j^\eps-P_k^\eps$ is never vanishing. To see that, it is enough to prove that if $Q_{jk}^\eps$ does not vanish on $\mathcal{A}_{n_{\eps}}$ then $P_j^\eps-P_k^\eps$ does not vanish on $\nu_\eps \mathcal{A}_{n_{\eps}}$ for $\nu_\eps>0$ large enough independent of $j,k$. In that case, one has, for $\nu>0$ large enough, $x\in S^1$ and $j,k\in  \{1,\dots,l\}$, 
\[(P_j^\eps-P_k^\eps)(\nu x) = \nu^{d'} \left(Q_{jk}^\eps(x) +o(1) \right),\]
where $d'$ is the positive degree of $Q_{jk}^\eps$ and $o(1)$ is a function of $x$ and $\nu$ tending to $0$ as $\nu$ tends to infinity uniformly with respect to $x\in S^1$ and $j,k\in  \{1,\dots,l\}$. This concludes the proof of the claim. 

Since the arcs $\mathcal{A}_1,\dots,\mathcal{A}_{N}$ 
do not depend on $\eps$, 
there exist one of them, denoted by ${\mathcal{A}}$, 
and sequences
$(\eps_m)_{m\in\mathbb{N}}, (\nu_m)_{m\in\mathbb{N}}$ in $(0,+\infty)$,
and $(V_m)_{m\in\mathbb{N}}$ in $\mathcal{P}^2_d$ such that $\lim_{m\to\infty}\eps_m=0$
and $V^{\eps_m}=V_m$ on $\nu {\mathcal{A}}$ for every $m\in\N$ and every $\nu\ge \nu_m$.
Let $\hat V_m\in \mathcal{P}^2_d$ be the homogeneous term of maximal degree of $V_m$.
Notice that
\[\hat V_m(x)=\lim_{\nu \to +\infty}\nu^{-d_m} V_m(\nu x),\qquad \forall x\in \R^2,\]
where $d_m$ denotes the degree of $V_m$.
Up to normalizing $V^{\eps_m}$,  we may assume that the maximum of the moduli of the coefficients of the polynomial $\hat V_j$ is equal to $1$. Thus, up to extracting a subsequence, $\hat V_m$ converges uniformly on compact sets to some nonzero 
 $\bar V\in\mathcal{P}^2_d$.

Similarly to the proof of Theorem~\ref{no-universal}, we can construct a periodic trajectory $t\mapsto  \Phi_{ A^0}(t,0)\bar x$ 
starting at $\bar x$ in the interior of the arc $ {\mathcal{A}}$, 
with $ A^0(\cdot)$ piecewise constant taking values in $\M^0$. 
Consider the switching laws $ A^\eps(\cdot)= A^0(\cdot)-\eps\Id_2$  taking  values in $\M^\eps$.
For every $t\ge 0$ such that  $\Phi_{ A^0}(t,0)\bar{x} \in{\mathcal{A}}$ and for every $m\in\N$, we have 
\begin{align*}
\hat V_{m}(\Phi_{ A^{\varepsilon_m}}(t,0)\bar x)
&=\lim_{\nu\to +\infty}\nu^{-d_m}V_{m}(\nu \Phi_{A^{\varepsilon_m}}(t,0)\bar x)\\
&=\lim_{\nu\to +\infty}\nu^{-d_m}V^{\eps_m}(\Phi_{ A^{\varepsilon_m}}(t,0)\nu \bar x)\\
&\le \lim_{\nu\to +\infty}\nu^{-d_m}V^{\eps_m}(\nu \bar x)\\
&=\lim_{\nu\to +\infty}\nu^{-d_m}V_{m}(\nu \bar x)=\hat V_{m}(\bar x),
\end{align*}
and therefore 
\begin{align*}
\bar V(\Phi_{A^0}(t,0)\bar x)&=\lim_{m\to\infty}\hat V_{m}(\Phi_{ A^{\varepsilon_m}}(t,0)\bar x)\\
&\le
\lim_{m\to\infty}\hat V_{m}(\bar x) = \bar V(\bar x).
\end{align*}
We then deduce that $t\mapsto \bar V(\Phi_{ A^0}(t,0)\bar{x}_0)$ is constant on $\{t\ge 0\mid \Phi_{A^0}(t,0)\bar{x}_0\in {\mathcal{A}}\}$. 
Moreover $\Phi_{A^0}(t,0) = e^{t M^0_i}$ for some $i=1,2$, for $t$ small enough, and by repeating the argument in the proof of Theorem~\ref{no-universal}  we obtain that $\bar V\equiv 0$, yielding a contradiction.
\end{proof}
\begin{remark}
At the light of the previous results, one may wonder if it is possible to identify universal classes of Lyapunov functions defined with a finite number of parameters.
The results proved in this section 
show that this is not the case for linear spaces of analytic functions and families of piecewise polynomials. 
In~\cite{bcm} an explicit construction of a universal class depending only on six parameters has been provided in the special case of two-dimensional linear switched systems with two modes. Unfortunately, since such a construction is based on an explicit characterization of the stability properties of the switched system in terms of the matrix coefficients, it seems unlikely that  it can be adjusted to higher dimensional switched systems.
\end{remark}

\bibliographystyle{abbrv}
\bibliography{biblio-switch}

 \end{document}